\documentclass[preprint]{elsarticle}

\usepackage{amsmath,amssymb}
\usepackage{url}

\newdefinition{definition}{Definition}
\newtheorem{proposition}{Proposition}
\newtheorem{theorem}{Theorem}
\newtheorem{lemma}{Lemma}
\newtheorem{corollary}{Corollary}

\newproof{proof}{Proof}

\begin{document}

\title{Remoteness of permutation codes}
\author[maths]{Peter J. Cameron}
\ead{p.j.cameron@qmul.ac.uk}

\author[ecs]{Maximilien Gadouleau\fnref{fn1}}
\ead{m.r.gadouleau@durham.ac.uk}

\address[maths]{School of Mathematical Sciences, Queen Mary, University of London, Mile End Road, London, E1 4NS}
\address[ecs]{School of Engineering and Computing Sciences, Durham University, South Road, Durham, DH1 3LE}

\fntext[fn1]{Supported by EPSRC ref:EP/H016015/1.}


\begin{abstract}
In this paper, we introduce a new parameter of a code, referred to as the remoteness, which can be viewed as a dual to the covering radius. Indeed, the remoteness is the minimum radius needed for a single ball to cover all codewords. After giving some general results about the remoteness, we then focus on the remoteness of permutation codes. We first derive upper and lower bounds on the minimum cardinality of a code with a given remoteness. We then study the remoteness of permutation groups. We show that the remoteness of transitive groups can only take two values, and we determine the remoteness of transitive groups of odd order. We finally show that the problem of determining the remoteness of a given transitive group is equivalent to determining the stability number of a related graph.
\end{abstract}
\maketitle

\section{Introduction} \label{sec:intro}

Let $X$ be a finite set of points and $d$ be a metric on $X$ which takes integral values. For any $v \in X$ and $t \ge 0$, we refer to the set $B_t(v) = \{u \in X : d(u,v) \le t\}$ as the ball of radius $t$ centered at $v$. We denote the minimum volume of a ball with radius $t$ in $X$ as $V^{\min}_t$ and the corresponding maximum as $V^{\max}_t$.

Let $C \subseteq X$, $C \ne \emptyset$ be a {\em code}, i.e. a set of points, which we will refer to as {\em codewords}. The maximum distance between any two codewords in $C$ is the {\em diameter} of $C$:
\begin{equation} \label{eq:delta} \nonumber
    \delta(C) = \max_{c,c' \in C} d(c,c'),
\end{equation}
while the minimum radius of a ball centered at a codeword needed to cover $C$ is the {\em radius} of $C$:
\begin{equation} \label{eq:rho} \nonumber
    \rho(C) = \min_{c \in C} \max_{c' \in C} d(c,c').
\end{equation}
It is well-known that the diameter and the radius are related by \cite[Ch. 6, Problem 10]{MS77}
\begin{equation}\label{eq:rho_delta} \nonumber
    \rho(C) \le \delta(C) \le 2\rho(C).
\end{equation}

Another important parameter of a code $C$ is its {\em covering radius}, defined as the minimum radius such that the balls centered around the codewords cover the whole set $X$:
\begin{equation} \label{eq:cr} \nonumber
    \mathrm{cr}(C) = \max_{v \in X} \min_{c \in C} d(v,c).
\end{equation}
For a thorough exposition of the covering radius, see \cite{CHLL97}.

\section{Results for all metric spaces} \label{sec:all_graphs}

\begin{definition}
For any code $C \subseteq X$, the {\em remoteness} of $C$ is defined as the minimum radius of a ball that covers the whole code:
\begin{equation} \label{eq:r(S)} \nonumber
    r(C) = \min_{v \in X} \max_{c \in C} d(v,c).
\end{equation}
\end{definition}

For any code $C$, we have
\begin{equation} \label{eq:trivial_r} \nonumber
    \frac{\delta(C)}{2} \le r(C) \le \rho(C).
\end{equation}
Furthermore, $r(\{v\}) = 0$ for all $v \in X$ and $r(X) = \rho(X)$.

Clearly, the maximum cardinality of a code with remoteness at most $t$ is given by $V^{\max}_t$. We are hence interested in the minimum cardinality of a code with remoteness at least $t$ for $t = 0,\ldots,\rho(X)$, which we denote as $m(X,d,t)$ henceforth, or simply
$$
    m(X,t) = \min_{r(C) = t} |C|.
$$
We have $m(X,0) = 1$, $m(X,t) = 2$ for $1 \le t \le \left\lceil \frac{\delta(X)}{2} \right\rceil$, and in general $m(X,t)$ is a non-decreasing function of $t$. The consideration above also shows that $m(X,t) \le V^{\max}_{t-1} + 1$; however this bound is usually very poor.

We now give a lower bound on the remoteness. Recall that an $(n,r,k)$-covering design is a family of $r$-subsets (called blocks) of a set of size $n$, where each $k$-set is contained in at least one block \cite{MM92}. We denote the minimum cardinality of an $(n,r,k)$-covering design as $K(n,r,k)$. A table of the tightest bounds on $K(n,r,k)$ known so far is available at \cite{Mar}. Denote the maximum remoteness of a code with cardinality $k$ as $r(k) = \max \{r(C) : |C| = k\}$; thus $r(k) = \max \{t: m(X,t) \le k\}$.

\begin{proposition} \label{prop:covering_design}
For all $v \in X$, let $B'_{r(k)}(v)$ be a set of $V^{\max}_{r(k)}$ points containing $B_{r(k)}(v)$. Then the family $\{B'_{r(k)}(v) : v \in X\}$ forms an $(|X|,V^{\max}_{r(k)},k)$-covering design and
\begin{equation} \label{eq:r_K} \nonumber
    r(k) \ge \min \{t : K(|X|,V^{\max}_t,k) \le |X|\}.
\end{equation}
\end{proposition}

\begin{proof}
By definition, for any code $C$ of $k$ codewords, there exists a point $v$ such that $C \subseteq B_{r(k)}(v) \subseteq B'_{r(k)}(v)$. Therefore, the collection $\{B'_{r(k)}(v)\}$ forms a covering design and $K(|X|,V^{\max}_{r(k)},k) \le |X|$.
\end{proof}

For any code $C$, we denote the number of points at distance no more than $t$ from all codewords as $\mu(t,C)$. Remark that $\mu(t,C) >0$ if and only if $t \le r(C)$. Then we have $|\mu(t-1,C)| + |C| \ge m(X,t)$ for $t \le \rho(X)$. This holds because for each element in $\mu(t-1,C)$ we could choose a point at distance at least $t$ from it. Adding these points to $C$ yields a set with remoteness at least $t$ and cardinality at most $|\mu(t-1,C)| + |C|$. Thus, $m(X,t)$ can be viewed as a lower bound on the intersection of balls.

In general, the problem of remoteness can be viewed as a special case of strong domination in graphs \cite{HHS98}. Recall that a strong dominating set (also referred to as total dominating set) in a graph is a set of vertices $C \subseteq V$ such that any vertex of the graph is adjacent to some element of $C$. The following proposition is easily seen.

\begin{proposition} \label{prop:strong_dominating}
For $0 \le t \le \rho(X)$, let $E_t = \{uv : u,v \in X, d(u,v) \ge t\}$ and define the graph $G_t = (X,E_t)$. Then $r(C) \ge t$ if and only if $C$ is a strong dominating set of $G_t$.
\end{proposition}

Since $m(X,t)$ is the solution of a special set cover problem \cite{GJ79}, we can apply the bounds derived for the general case. We obtain \cite{CHLL97}
\begin{equation} \label{eq:bounds_m} \nonumber
    \frac{n}{n - V^{\min}_{t-1} + 1} \le m(X,t) \le \frac{n}{n - V^{\min}_{t-1} + 1} + \frac{n}{n - V^{\max}_{t-1} + 1}\ln (n - V^{\min}_{t-1} + 1).
\end{equation}

The lower bound is usually very poor, as we need $V^{\min}_{t-1} > \frac{n}{2} + 1$ to make it non-trivial. A code with a cardinality no more than the upper bound can be obtained by using a greedy algorithm \cite{Sla97}. The upper bound can be further refined by the techniques in \cite{CD97}.

The remoteness is closely related to the covering radius, as seen in Proposition~\ref{prop:r_cr} below.

\begin{proposition} \label{prop:r_cr}
For any code $C$,
\begin{equation} \label{eq:r_cr} \nonumber
    \rho(X) \le r(C) + \mathrm{cr}(C) \le \rho(X)+ \delta(X).
\end{equation}
\end{proposition}

\begin{proof}
The upper bound is trivial, we now prove the lower bound. It suffices to show that for any point $v \in X$, there exists a codeword in $C$ at distance at least $\rho(X) - \mathrm{cr}(C)$ from $v$. For any $v$, there exists $u$ such that $d(u,v) \ge \rho(X)$, and by definition of the covering radius there exists $c \in C$ such that $d(u,c) \le \mathrm{cr}(C)$. Hence the triangle inequality implies $d(v,c) \ge \rho(X) - \mathrm{cr}(C)$.
\end{proof}

Note that the bounds in Proposition~\ref{prop:r_cr} can be tight. For instance, if $C = X$, then $r(C) + \mathrm{cr}(C) = \rho(X)$. On the other hand, a pair of leaves in the star graph with at least $4$ vertices satisfies $r(C) = 1 = \rho(X)$, while $\mathrm{cr}(C) = 2 = \delta(X)$.

Denoting the minimum cardinality of a code with covering radius $t$ as $M_{\mathrm{cr}}(X,t)$, Proposition~\ref{prop:r_cr} implies
$m(X,t) \le M_{\mathrm{cr}}(X,\rho(X) - t)$.

Furthermore, we say that the metric space $(X,d)$ is {\em balanced} if $\rho(X) = \delta(X)$ and if for any $v \in X$, there exists $\bar{v} \in X$ at distance $\rho(X)$ such that
$$
    d(u,v) + d(u,\bar{v}) = \rho(X)
$$
for all $u \in X$. For instance the binary Hamming graph $H(n,2)$ (the $n$-dimensional hypercube) with the shortest path distance is a balanced metric space, where $\bar{v} = v + 1^n$, the all-ones vector.

\begin{corollary} \label{cor:balanced}
If $X$ is balanced, then $r(C) + \mathrm{cr}(C) = \rho(X)$ for any code $C \subseteq X$. Therefore, $m(X,t) = M_{\mathrm{cr}}(X,\rho(X)-t)$.
\end{corollary}

\begin{proof}
There exists $v \in X$ such that $d(v,c) \geq {\rm cr}(C)$ for all $c \in C$, with equality being reached for some codeword in $C$. Then we have $d(\bar{v},c) \leq \rho(X) - {\rm cr}(C)$ for all $c \in C$ by the triangle inequality. Hence $r(C) \leq \rho(X) - {\rm cr}(C)$.
\end{proof}

\section{Remoteness of permutation codes} \label{sec:permutations}

We now consider $X = S_n$ the symmetric group on the first $n$ natural integers, where the distance between two permutations is the Hamming distance: $d(\pi,\sigma) = |\{i: i\pi \ne i\sigma\}|$ for any $\pi,\sigma \in S_n$. Remark that the Hamming distance is invariant under left and right translation: $d(\pi,\sigma) = d(\tau\pi,\tau\sigma) = d(\pi\tau,\sigma\tau)$ for all $\pi,\sigma,\tau \in S_n$. Note that $d(\pi,\sigma) \in \{0\} \cup \{2,3,\ldots,n\}$. A permutation $\pi \in S_n$ can be represented in passive form as a word in $\{1,\ldots,n\}^n$ with coordinates  $1\pi,2\pi, \ldots, n\pi$. We will also use the notation $J\pi = \{j\pi : j \in J\}$ for any set $J \subseteq \{1,\ldots,n\}$.

Subsets of the symmetric group, referred to as {\em permutation codes}, have been intensively studied recently (see the thorough survey in \cite{Cam10} and references therein). In particular, the covering radius of permutation codes has been investigated in \cite{CW05}.

\subsection{Preliminary results} \label{sec:permutations_misc}

First of all, let us consider the remoteness of any pair of permutations. If they are at distance $2$, then the remoteness is clearly $2$. However, when the distance increases, the remoteness may vary for pairs of permutations with the same distance. By translation, we only consider pairs of the form $C = \{(1),\sigma\}$, where $(1)$ denotes the identity. The remoteness of $C$ depends on the cycle structure of $\sigma$, denoted as $T_1,T_2,\ldots,T_k \subseteq \{1,\ldots,n\}$ of respective lengths $l_1,\ldots,l_k$, where $l_c \ge 2$ and $\sum_{c=1}^k l_c = d((1),\sigma) = d$ and $\sigma$ reduces to a cyclic permutation $\sigma_c$ of $T_c$ for all $c$. We are interested in finding a permutation $\pi \in S_n$ which minimizes $\max \{d(\pi,(1)), d(\pi,\sigma)\}$, which we will refer to as a {\em minimal permutation}. Let us first focus on the case where $k=1$, i.e. $\sigma$ is a cyclic permutation.

\begin{lemma} \label{lemma:cyclic_pair}
If $\kappa \in S_n$ is a cyclic permutation and $\pi \in S_n \backslash C$, then
$$
    d(\pi,(1)) + d(\pi,\kappa) \ge n+1.
$$
Conversely, for all $2 \le e \le n-1$, there exists $\tau_e \in S_n$ such that $d(\tau_e,(1)) = e$ and $d(\tau_e,\kappa) = n+1-e$.
\end{lemma}

\begin{proof}
It is clear that $d(\pi,(1)) + d(\pi,\kappa) \ge n$ for any $\pi \in S_n$ by the triangle inequality. Suppose $\pi \notin C$ satisfies $d(\pi,(1)) + d(\pi,\kappa) = n$ and let $S =\{j: j\pi = j\}$ be the set of indices on which $\pi$ agrees with the identity, and $\bar{S} = \{1,\ldots,n\} \backslash S$ be the set of indices on which $\pi$ agrees with $\kappa$. Since $\pi \notin C$, we have $\bar{S} \notin \{\emptyset, \{1,\ldots,n\}\}$ and hence there exists $j$ in $\bar{S}\kappa \cap S$. Thus $j\pi = j$ and $j = i\kappa = i\pi$ for some $i \in \bar{S}$ and hence $j = i$ which contradicts the fact that $S$ and $\bar{S}$ are disjoint.

Assuming $\kappa$ is the standard cyclic permutation, $j\kappa = j+1$ for $j \le n-1$ and $n\kappa = 1$. Then define $\tau_e$ as $j\tau_e = j$ for $1 \le j \le n-e$, $j\tau_e = j+1$ for $n-e+1 \le e \le n-1$ and $n\tau_e = n-e+1$. It is easily seen that $\tau_e$ satisfies the claim.
\end{proof}

Note that if $\pi \in C = \{(1),\sigma\}$, then $d((1),\pi) + d(\sigma,\pi) = n$. Hence Lemma~\ref{lemma:cyclic_pair} indicates that we can either minimize the sum of distances between $\pi$ and the pair of codewords (if $\pi \in C$) or try to balance the distances (otherwise) with an additional penalty of $1$ unit of distance. The strategy to obtain a minimal permutation for the general case is hence to pay the minimal amount of penalties. This amount is no more than one, and can even be zero under certain circumstances.

\begin{proposition} \label{prop:lower_r_permutations}
Suppose $d = d((1),\sigma)$ is even, and that we can order the cycle lengths $l_1,\ldots,l_k$ such that there exists $s$ for which $\sum_{c=1}^s l_c = \sum_{c=s+1}^k l_c =\frac{d}{2}$. Then $r(C) = \frac{d}{2}$. Otherwise, $r(C) = \left\lfloor \frac{d}{2} \right\rfloor + 1$. Thus $m(S_n,t) = 2$ for $2 \le t \le \left\lfloor\frac{n}{2}\right\rfloor + 1$.
\end{proposition}

\begin{proof}
First of all, we have $r(C) \ge \left\lceil \frac{d}{2} \right\rceil$. We now prove that there exists a minimal permutation $\pi$ such that $T_c\pi = T_c$ for all $1 \le c \le k$. Suppose the contrary, i.e. for any minimal permutation $\tau$, there exists a nonempty set of indices $c$ for which $T_c \ne T_c\tau$. For all such $c$, denote the elements of $T_c$ mapped outside of $T_c$ as $\{t_{c,1},\ldots,t_{c,m_c}\}$ and the elements outside of $T_c$ mapped into $T_c$ as $\{s_{c,1},\ldots, s_{c,m_c}\}$. Remark that $t_{c,j}\tau \notin \{t_{c,j},t_{c,j}\sigma\}$ and $s_{c,j}\tau \notin \{s_{c,j},s_{c,j}\sigma\}$ for all $c$ and $j$. Construct the permutation $\tau'$ as $s_{c,j}\tau' = t_{c,j}\tau$, $t_{c,j}\tau' = s_{c,j}\tau$ for all $c,j$ and $a\tau' = a\tau$ for any other $a \in \{1,\ldots,n\}$. Then it is readily checked that $\tau'$ is also minimal, while $T_c = T_c\tau$ for all $c$, which contradicts our hypothesis.

Therefore, $\pi$ can be decomposed into permutations $\pi_1,\ldots,\pi_k$ of $T_1,\ldots,T_k$ respectively. If the assumptions of the first sentence are satisfied, then simply let $\pi_c$ be the identity for $c \le s$ and $\pi_c = \sigma_c$ for $c \ge s+1$, then it is clear that $\pi$ is at distance $\frac{d}{2}$ from both codewords.

Otherwise, a penalty has to be paid, for if $\pi_c$ is the identity for $c \in J$ and $\pi_c =\sigma_c$ for all $c \notin J$ for some set of indices $J \subseteq \{1,\ldots,k\}$, then $d((1),\pi) = \sum_{c \in J} l_c \ne \frac{d}{2}$ (similarly for $\sigma$) while $d((1),\pi) + d(\sigma,\pi) = d$, and the maximum between the two distances is greater than $\frac{d}{2}$. Let us pay the penalty only once: let $1 \le s \le k$ such that $\sum_{c=1}^{s-1} l_c < \frac{d}{2}$ while $\sum_{c=1}^s l_c > \frac{d}{2}$. Again, we let $\pi_c$ be the identity for $c \le s-1$ and $\pi_c = \sigma_c$ for $c \ge s+1$. Denoting $e = \sum_{c=1}^s l_c - \left\lceil \frac{d}{2} \right\rceil + 1$, we then let
$$
    \pi_s = \begin{cases}
        \sigma_s & \mbox{if } \sum_{c=1}^{s-1} l_c = \left\lceil\frac{d}{2}\right\rceil - 1\\
        (1) & \mbox{if } \sum_{c=1}^s l_c = \left\lfloor\frac{d}{2}\right\rfloor + 1\\
        \tau_e & \mbox{otherwise, i.e. if } 2 \le e \le l_s-1.
    \end{cases}
$$
Then it is easy to check that  $\max\{d((1),\pi),d(\sigma,\pi)\} \le \left\lfloor \frac{d}{2} \right\rfloor + 1$.
\end{proof}

We are now interested in a refinement of the $r(C) + \mathrm{cr}(C) \ge n$ bound.

\begin{proposition} \label{prop:strict_triangular}
If $C$ is neither a singleton nor the whole symmetric group $S_n$, then
$$
    r(C) + \mathrm{cr}(C) \ge n +1.
$$
\end{proposition}

\begin{proof}
First of all, some trivial cases have to be dealt with. The statement is easily verified for $n \le 3$; let us assume $n \ge 4$. Then, if $\mathrm{cr}(C) = n$ either $r(C) = 0$ and hence $C$ is a singleton or $r(C) + \mathrm{cr}(C) \ge n+2$. Similarly if $r(C) = n$ then either $\mathrm{cr}(C) = 0$ hence $C = S_n$ or $r(C) + \mathrm{cr}(C) \ge n+2$.

Let us then assume that $C$ is a code in $S_n$ with $\mathrm{cr}(C) \le n-1$ and $r(C) \le n-1$. Remark that for any $\pi\in S_n$, we can construct a cyclic permutation $\kappa$ at distance $n$ from $\pi$ as follows. Let us express $\pi$ in cycle decomposition: $\pi = (1\, a_2 \, a_3\, \ldots)\ldots(a_j\,\ldots\, a_n)$, then $\kappa = (1\, a_3\, \ldots a_{n-1}\, a_2\, a_4\, \ldots a_n)$ if $n$ is even or $\kappa = (1\, a_3\, \ldots\, a_n\, a_2\, a_4\, \ldots\, a_{n-1})$ if $n$ is odd.

Therefore, for any $\pi \in S_n$, $C$ does not contain the set of all cyclic permutations multiplied by $\pi$ (as such a set has remoteness $n$). There hence exists a cyclic permutation $\kappa$ for which $\pi\kappa \notin C$. Due to the covering radius of $C$, there exists $c \in C$ such that $d(c,\pi \kappa) \le \mathrm{cr}(C)$ and thus $d(c,\pi) \ge n+1-\mathrm{cr}(C)$ by Lemma~\ref{lemma:cyclic_pair} (for $c \ne \pi\kappa$ because $\pi\kappa \notin C$ and $c \ne \pi$ because $d(c,\pi) \ge n-\mathrm{cr}(C) > 0$).
\end{proof}

Equality in Proposition~\ref{prop:strict_triangular} is achieved by many a code, e.g. any ball with radius $r$, with $2 \le r \le n-1$.

\subsection{Bounds} \label{sec:bounds_permutations}

Let us derive a lower bound on $m(S_n,t)$.

\begin{proposition} \label{prop:lower_m}
For $t \le n$, we have
\begin{equation} \label{eq:lower_m} \nonumber
    m(S_n,t) \ge \left\lfloor \frac{2n-t+1}{2(n-t+1)} \right\rfloor + 1.
\end{equation}
\end{proposition}

\begin{proof}
Let $\mu = \left\lfloor \frac{2n-t+1}{2(n-t+1)} \right\rfloor$ and let $C = \{c_1,c_2,\ldots,c_\mu\}$ be a code of $\mu$ permutations. We construct a permutation $\pi$ at Hamming distance at most $t-1$ from all codewords in $C$ recursively as follows. Let $I_1 = \emptyset$, and for all $1 \le j \le \mu$, let $A_j$ be a set of cardinality $n-t+1$ such that $A_j \cap I_j = \emptyset$ and $A_j c_j \cap I_j \pi = \emptyset$; then set $a\pi = ac_j$ for any $a \in A_j$ and update $I_{j+1} = I_j \cup A_j$. Finally, denote $A_{\mu + 1} = \{1,\ldots,n\} \backslash I_{\mu+1}$ and its elements as $\{a_1,\ldots,a_l\}$ and $\{1,\ldots,n\} \backslash I_{\mu+1} \pi = \{b_1,\ldots,b_l\}$; then let $a_i\pi = b_i$ for all $1 \le i \le l$.

We first verify that $A_j$ exists for all $1 \le j \le \mu$. This is done by recursion, where the initial step $j=1$ is trivial. Assume it is true up to $j-1$; We have
\begin{equation} \label{eq:I_j} \nonumber
    |I_jc_j \cup I_j\pi| \le |I_jc_j| + |I_j\pi| = 2(j-1)(n-t+1) \le t-1.
\end{equation}
Therefore, there exists a set $B_j$ of cardinality $n-t+1$ which does not intersect $I_jc_j \cup I_j\pi$. Let $A_j = B_jc_j^{-1}$, then $|A_j| = n-t+1$, $A_j \cap I_j = \emptyset$, and $A_jc_j \cap I_j\pi = \emptyset$. Finally, $A_{\mu+1}$ is well defined, for $|I_{\mu+1}| \leq n$.

Second, we verify that $\pi$ is indeed a permutation by considering two distinct numbers $1 \le a < b \le n$. Either $a,b \in A_j$ for some $j$ and $a\pi = ac_j \ne bc_j = b\pi$; or $a \in A_i$ and $b \in A_j$ for some $i \ne j$ and hence $b\pi \notin A_i\pi$, from which $b\pi \ne a\pi$.
\end{proof}

Let us now design a code with high remoteness by using rows of a Latin square. Recall that a Latin square of order $n$ is an $n \times n$ array over $\{1,\ldots,n\}$ such that any element of $\{1,\ldots,n\}$ appears in each row and each column \cite[Ch. 6]{Cam94}. The cyclic Latin square has as first row the elements $1$ to $n$ in increasing order, and each row is obtained from the previous one by a cyclic shift to the left. In other words, the rows of the cyclic Latin square are the passive forms of the elements of the group generated by the standard cyclic permutation.

\begin{proposition} \label{prop:partial_latin}
Let $C$ be the first $k$ rows of a Latin square of order $n$, then
$$
    r(C) \ge n - \left\lfloor \frac{n}{k} \right\rfloor.
$$
Furthermore, if $k \equiv 0 \mod 2$ and $n \equiv 0 \mod k$ and $C$ consists of the first $k$ rows of the cyclic Latin square of order $n$, then $r(C) \ge n - \frac{n}{k} + 1$. We obtain $m(t) \le \frac{n}{n-t+1}$ if $n-t+1 \, | \, \frac{n}{2}$ and $m(t) \le \left\lfloor \frac{n}{n-t+1} \right\rfloor + 1$ if $n-t+1 \, \nmid \, \frac{n}{2}$ and $t \le n-1$.
\end{proposition}

\begin{proof}
Let $c_i$ denote the $i$-th row of a Latin square. For any $\pi \in S_n$, we have $\sum_{i=1}^n d(\pi,c_i) \ge n(k-1)$ and hence there exists $c_j$ such that $d(\pi,c_j) \ge \frac{n(k-1)}{k}$. This proves the first result. Now, let $k \equiv 0 \mod 2$ and $n \equiv 0 \mod k$ and let $C$ be the first $k$ rows of the cyclic Latin square of order $n$. Suppose there exists $\pi \in S_n$ such that $d(\pi,c_i) \le \frac{n(k-1)}{k}$ for all $1 \le i \le k$. By the argument above, there actually has to be equality for all $i$, hence for all $1 \le j \le k$, there exists $i_j$ such that $j\pi = jc_{i_j}$. Denoting the content of the $(i,j)$ cell of the Latin square as $L(i,j)$ and $\Delta(i,j) = L(i,j) - i - j + 1 \mod n$, then $\Delta$ is identically zero on the cyclic Latin square. We have
$$
    0 = \sum_{j=1}^n \Delta(i_j,j) \equiv \sum_{z=1}^n z - \frac{n}{k}\sum_{i_j=1}^k i_j - \sum_{j=1}^n j \mod n \equiv - n\frac{k+1}{2} \mod n,
$$
which contradicts the fact that $k$ is even. Thus, $r(C) \ge \frac{n(k-1)}{k} + 1$.
\end{proof}

A {\em transversal} in a Latin square of order $n$ is a collection of $n$ positions of the square comprising one from each row and one from each column, such that the symbols in those positions are distinct. A transversal can hence be viewed as a permutation $\pi \in S_n$ at Hamming distance $n-1$ from all rows.

\begin{corollary} \label{cor:cyclic_latin}
The set of rows of a Latin square has remoteness $n-1$ if it has a transversal and remoteness $n$ otherwise. We obtain $m(S_n,n-1) \le n$ for all $n$, and if $n$ is even then $m(S_n,n) \le n$.
\end{corollary}

For $n$ odd, the case of the full remoteness is not covered by our constructions based on Latin squares. However, we can add more codewords to a Latin square to reach a remoteness of $n$. For $n \ge 5$, \cite{WW06} indicates that there exists a Latin square of order $n$ (referred to as a confirmed bachelor) which contains an entry through which no transversal passes.

\begin{proposition} \label{prop:extended_latin}
Let $n \ge 5$ be odd and let $C \subseteq S_n$ consist of the rows of a confirmed bachelor Latin square of order $n$. Let $D = \{(2i-1 \, 2i): 1 \le i \le \frac{n-1}{2}\}$, then $r(C \cup D) = n$. Therefore, $m(S_n,n) \le \frac{3n-1}{2}$.
\end{proposition}

\begin{proof}
By permutation of rows and columns, let us assume that the entry through which no transversal passes is $(n,n)$. Also, by renaming entries, we can assume that the first row of the confirmed bachelor Latin square is the identity. If $\pi \in S_n$ is not a transversal of that Latin square, then there exists a row of the Latin square at distance $n$ from it. Otherwise, $\pi$ agrees with the identity in exactly one position, say $j \leq n-1$: we have $j\pi = j$ and $i\pi \ne i$ for $i \ne j$. Then it is easily checked that $d(\pi, (2k-1 \, 2k)) = n$, where $j \in \{2k-1, 2k\}$.
\end{proof}

Now let us consider the cartesian product of two codes. Let $C_1, C_2$ be two codes in $S_{n_1}$ and $S_{n_2}$, respectively. Their cartesian product $C = C_1 \times C_2 \subseteq S_{n_1 + n_2}$ is the set of permutations $c = (c_1,c_2)$, where $ic = ic_1$ for $1 \le i \le n_1$ and $ic = (i-n_1)c_2 + n_1$ for $n_1 + 1 \le i \le n_2$.

\begin{proposition} \label{prop:cartesian}
For all $C_1, C_2$, we have
\begin{equation} \label{eq:cartesian} \nonumber
    r(C_1 \times C_2) = r(C_1) + r(C_2).
\end{equation}
\end{proposition}

\begin{proof}
For all $\pi = (\pi_1,\pi_2) \in S_{n_1 + n_2}$ and $c \in C$, we have $d(\pi,c) = d(\pi_1,c_1) + d(\pi_2,c_2)$. Thus for any $\pi \in S_{n_1 + n_2}$, there exists a codeword at distance at least $r(C_1) + r(C_2)$ from $\pi$; conversely, if $\pi_1$ and $\pi_2$ are minimal for $C_1$ and $C_2$ respectively then $d(\pi,c) = r(C_1) + r(C_2)$.
\end{proof}

\begin{corollary} \label{cor:fixed_points}
If all codewords $\pi \in C$ satisfy $i_l\pi = j_l$ for $l=1,\ldots,k$, then by translation we may assume $i_l = j_l$ for $l = 1,\ldots,k$; the remoteness is unaffected by restriction to $\{1,\ldots,n\} \backslash \{i_1,\ldots,i_k\}$.
\end{corollary}

The remoteness satisfies some inequalities analogous to the Singleton bound for the minimum distance of codes.

\begin{proposition} \label{prop:m(n-1)}
We have
\begin{equation} \label{eq:m(n-1)} \nonumber
    m(S_{n-1},t-2) \le m(S_n,t) \le m(S_{n-1},t).
\end{equation}
\end{proposition}

\begin{proof}
For any $\pi \in S_n$, we define the permutation $\pi_1 \in S_{n-1}$ as $i\pi_1 = i\pi$ for all $i \in \{1,\ldots,n-1\} \backslash \{n\pi^{-1}\}$, and $(n\pi^{-1})\pi_1 = n\pi$ if $n\pi^{-1} \ne n$. Conversely, any permutation in $S_{n-1}$ can be expressed as $\pi_1$ for some $\pi \in S_n$ fixing $n$.

We first prove the upper bound. Let $C_1 \subseteq S_{n-1}$ be a code with remoteness at least $t$ and cardinality $m(S_{n-1},t)$ and let $C = \{c \in S_n : c_1 \in C_1\}$. We shall prove that $r(C) \ge t$ by considering any permutation $\pi \in S_n$. If $\pi$ fixes $n$, then $d(\pi,c) = d(\pi_1,c_1)$ for all $c$ and hence there exists $c$ such that $d(\pi,c) \ge t$. Otherwise, there exists $c$ such that $d(\pi,c) \ge 1 + d(\pi_1,c_1) \ge t+1$. Thus $m(S_n,t) \le |C| = |C_1| = m(S_{n-1},t)$.

We now prove the lower bound. Let $C \subseteq S_n$ be a code with remoteness $t$ and cardinality $m(S_n,t)$. For any $\pi_1 \in S_{n-1}$, we have $d(\pi_1,c_1) \ge d(\pi,c) - 2$. Let $C = \{c_1: c \in C\}$, then $C$ has remoteness at least $t-2$ and hence $m(S_{n-1},t-2) \le |C_1| \le |C| = m(S_n,t)$.
\end{proof}

By using the passive form, a permutation in $S_n$ can be viewed as a word in the Hamming graph $H(n,n)$. It immediately follows that if $C$ is a permutation code, then $r(S_n,C) \ge r(H(n,n),C)$. On the other hand, some codewords can be added to $C$ to produce a remote code for the Hamming graph.

\begin{proposition} \label{prop:m(H)<m(S)}
We have
\begin{equation} \label{eq:m(H)<m(S)} \nonumber
    m(H(n,n),n) \le n + m(S_n,n).
\end{equation}
\end{proposition}

\begin{proof}
Let $C \subseteq S_n$ be a permutation code with remoteness $n$ and let $D \subseteq H(n,n)$ be defined as $D = \{d_a = (a,a,\ldots,a): 1 \le a \le n\}$. Viewing the permutations in $C$ in passive form as words in $H(n,n)$, we shall prove that $r(C \cup D) = n$. Let $v \in H(n,n)$; there are two cases. First, if $v$ also represents a permutation, then there exists a codeword in $C$ at distance $n$ from $v$. Otherwise, a coordinate value $a$ is not on any coordinate of $v$, and $d(v,d_a) = n$.
\end{proof}


\section{Remoteness of permutation groups} \label{sec:groups}

\subsection{Groups generated by one element}

Let us now consider the remoteness of a group $G$ generated by one element $g$. In view of Corollary~\ref{cor:fixed_points}, we assume that $g$ has no fixed points. Let $T_1,\ldots,T_k$ of lengths $l_1,\ldots,l_k$ denote the cycle decomposition of $g$. Then it is easily seen that $r(G) \ge n-k$ and that $r(G) = n-k$ if and only if there is a permutation $\pi \in S_n$ at distance $n-k$ from all the elements of $G$.

\begin{theorem} \label{th:one_generator}
If $g$ is an even permutation, then $r(\langle g \rangle) = n-k$; if $g$ is an odd permutation, then $r(\langle g \rangle) = n-k+1$.
\end{theorem}

\begin{proof}
We first prove the following claim: Let $\kappa$ be a cyclic permutation on $\{0,\ldots,2m-1\}$, then there exist $\pi_0, \pi_1 \in S_{2m}$ such that
\begin{eqnarray*}
    d(\pi_0,\kappa^a) &=& \left\lbrace \begin{array}{ll} 2m - 2 & \mbox{if } a \mbox{ is even}\\ 2m & \mbox{if } a \mbox{ is odd}, \end{array} \right.\\
    d(\pi_1,\kappa^a) &=& \left\lbrace \begin{array}{ll} 2m & \mbox{if } a \mbox{ is even}\\ 2m - 2 & \mbox{if } a \mbox{ is odd}. \end{array} \right.
\end{eqnarray*}

We assume that $\kappa$ is the standard cyclic permutation, i.e. $i\kappa^a = i + a$ (all operations are modulo $2m$). We differentiate on the parity of $m$. If $m$ is even, then let $i\pi_0 = 3i$ and $i\pi_1 = 3i+1$ for all $0 \leq i \leq 2m-1$. We prove that $\pi_0$ is indeed a permutation at distance $2m-2$ from all even powers of $\kappa$ and at distance $2m$ from all odd powers. First, $\pi_0$ is a permutation since $3i = 3j$ implies $i=j$. Second, $i\kappa^a = i\pi_0$ if and only if $a = 2i$: all even values of $a$ occur twice. The proof for $\pi_1$ is similar. If $m$ is odd, let $i\pi_0 = i + 2 \left\lfloor \frac{i}{2} \right\rfloor$ and $i\pi_1 = i\pi_0 + 1$ for all $i$. It is easily checked that $\pi_0$ and $\pi_1$ are permutations satisfying the claims.

We now prove that $r(G) = n-k$ if $g$ is even and $r(G) \le n-k+1$ if $g$ is odd. Let $g$ be an even permutation, i.e. there is an even number of even cycle lengths, say $2s$. Let us construct a permutation $\pi$ such that $d(\pi,g^a) = n-k$ for all $0 \le a \le |G|-1$. Let $\pi$ reduce to $\pi_0$ for the first $s$ cycles of even lengths, to $\pi_1$ for the other $s$ cycles of even lengths, and let $\pi$ reduce to a transversal of the cyclic Latin square for all the $k-2s$ odd cycles. Since $g^a$ reduces to $\kappa^\lambda$, where $\kappa$ is a cyclic permutation and $\lambda \equiv a \mod l_c$, on $T_c$ for all $1 \le c \le k$, we have
\begin{equation} \nonumber
    d(\pi,g^a) = \sum_{c=1}^s d(\pi_0,\kappa^a) + \sum_{c=s+1}^{2s} d(\pi_1,\kappa^a) + \sum_{c=2s+1}^k (l_c-1) = \sum_{c=1}^k l_c - k = n-k.
\end{equation}
If $g$ is an odd permutation, then without loss $l_1$ is even and the restriction $g'$ of $g$ on $T_2 \cup \ldots \cup T_k$ is even. Therefore, there exists $\pi'$ at distance $n-l_1-k+1$ from all the powers of $g'$, and by extension there exists $\pi \in S_n$ at distance at most $n-k+1$ from all powers of $g$.

Let us finally prove that when $g$ is an even permutation, then there is no permutation $\pi$ such that $d(\pi,g^i) = n-k$ for all $0 \le i \le |G| - 1$. We show it by contradiction. First, by an argument similar to that for Proposition~\ref{prop:lower_r_permutations}, we can show that there is always a minimal permutation $\pi$ which restricts to a permutation on all $T_c$'s. Let $\chi(i,j) = 1$ if $jg^i = j\pi$ and $\chi(i,j) = 0$ otherwise. Thus
\begin{equation} \label{eq:sum1} 
    \sum_{j=0}^{n-1} \sum_{i=0}^{|G| - 1} i \chi(i,j)= k \sum_{i=0}^{|G| - 1} i = k |G| \frac{|G|-1}{2}.
\end{equation}
For any cycle $T_c$ of length $l_c$, denote $m_c = \frac{|G|}{l_c}$. For all $j \in T_c$, we can express $\{i: j\pi = jg^i\}$ as $\{i' + al_c: 0 \le a \le m_c-1\}$ for some $i'$ with $0 \leq i' \leq l_c-1$. We obtain
\begin{eqnarray}
    \nonumber
    \sum_{j=0}^{n-1} \sum_{i=0}^{|G| - 1} i\chi(i,j) &=& \sum_{j=0}^{n-1} \sum_{a=0}^{m_c-1} (i' + a l_c)\\
    \nonumber
    &=& \sum_{j=0}^{n-1} i' m_c + \sum_{j=0}^{n-1} |G| \frac{m_c-1}{2}\\
    \label{eq:sum2}
    &=& \sum_{j=0}^{n-1} i' m_c + k \frac{|G|^2}{2} - n\frac{|G|}{2}.
\end{eqnarray}
Combining (\ref{eq:sum1}) and (\ref{eq:sum2}), we obtain
\begin{equation} \label{eq:sum22}
    \sum_{j=0}^{n-1} i' m_c = |G| \frac{n-k}{2}.
\end{equation}

On the other hand, for all $j \in T_c$, denote $j' = j - \sum_{b=1}^{c-1} l_b$ (so that $j'$ ranges from $0$ to $l_c-1$) and $Z_j = (j' + i') \mod l_c = j\pi - \sum_{b=1}^{c-1} l_b$. Note that $Z_j$ also ranges from $0$ to $l_c-1$. Thus,
\begin{equation} \label{eq:sum3}
    \sum_{j=0}^{n-1} m_c Z_j = \sum_{j=0}^n m_c j' = \sum_{c=1}^k m_c \sum_{j'=0}^{l_c-1} j' = \sum_{c=1}^k |G| \frac{l_c-1}{2} = |G|\frac{n-k}{2}.
\end{equation}
Finally, combining (\ref{eq:sum22}) and (\ref{eq:sum3}) yields
$$
|G| \frac{n-k}{2} = \sum_{j=0}^{n-1} m_c Z_j \equiv \sum_{j=0}^{n-1} m_c (i' + j') \mod |G| = |G|(n-k) \equiv 0 \mod |G|.
$$
Therefore, $n-k$ is even. However, this is equivalent to: there are an even number of cycle lengths. Indeed, denote the number of even-length cycles as $E$ and that of odd cycles as $O$, where $E+O = k$. We have $n \equiv O \mod 2$ and hence $n-k \equiv O - (E + O) \equiv E \mod 2$.
\end{proof}

\subsection{Transitive groups}

\begin{proposition} \label{prop:transitive}
A transitive group has remoteness $n-1$ if and only if it has covering radius $n-1$; otherwise, it has remoteness $n$.
\end{proposition}

\begin{proof}
Let $G$ be a transitive group; we know that $\mathrm{cr}(G) \le n-1$ by \cite[Proposition 15]{CW05}. By the coset version of the orbit-counting lemma, the average distance between any permutation $\pi \in S_n$ and to $G$ is $n-1$. Therefore, $r(G) \ge n-1$ with equality if and only if there exists $\pi \in S_n$ such that $d(\pi,g) = n-1$ for all $g \in G$ and hence $\mathrm{cr}(G) \ge n-1$.
\end{proof}

\begin{corollary} \label{cor:2-transitive}
Any $2$-transitive group has remoteness $n$. If $G$ acts regularly, then the Hall-Paige conjecture \cite{HP55} proved in \cite{Bra,Eva09,Wil09} implies that $r(G) = n-1$ if and only if its Sylow $2$-subgroup is non-cyclic.
\end{corollary}

We remark that if a transitive permutation group $G$ has remoteness $n-1$, then for any $1 \le i \le n$, ${\rm Stab}_G(i)$, acting on the remaining $n-1$ points, has covering radius $n$.

We have shown in Proposition~\ref{prop:partial_latin} that the remoteness of the cyclic group $C_n$ acting on $n$ elements has remoteness $n-1$ when $n$ is odd and remoteness $n$ when $n$ is even. The dihedral group $D_{2n}$ is treated in the next proposition.

\begin{proposition} \label{prop:dihedral}
We have $r(D_{2n}) = n-1$ if $n$ is congruent to $1$ or $5$ modulo $6$ and $r(D_{2n}) = n$ otherwise.
\end{proposition}

\begin{proof}
For ease of presentation, assume that $D_{2n}$ acts on $\mathbb{Z}_n$. The whole dihedral group can be viewed as two Latin squares: the cyclic Latin square formed by elements $\kappa^a$ for $0 \le a \le n-1$, where $j\kappa^a = a + j$, and a second Latin square formed by elements $\sigma\kappa^b$ for $0 \le b \le n-1$ where $j\sigma\kappa^b = b-j$. First, $C_n \le D_{2n}$, so $r(D_{2n}) \ge n-1$ for all $n$ and $r(D_{2n}) = n$ for $n$ even. Second, for $n$ odd and not a multiple of $3$, we can easily show that the permutation $\pi$ defined as $j\pi = 2j$--i.e., the diagonal of the cyclic Latin square--is at distance $n-1$ from all the elements of the dihedral group (it agrees with $\kappa^j$ and $\sigma\kappa^{3j}$ on position $j$). Third, let $n$ be an odd multiple of $3$, and suppose that there exists $\pi \in S_n$ at distance $n-1$ from all the elements of $D_{2n}$. It agrees on position $j$ with $\kappa^{a_j}$ where $a_j=j\pi-j$, and with $\sigma\kappa^{b_j}$ where $b_j = j\pi+j$. Denoting the square pyramidal number $P = \sum_{j=0}^{n-1} j^2 = n\frac{(n-1)(2n-1)}{6}$, we have $\sum_{j=0}^{n-1} (2j)^2 \equiv P \mod n$ and hence
\begin{eqnarray*}
    P &\equiv& \sum_{j=0}^{n-1} 4(j\pi)^2 \mod n\\
    &\equiv& \sum_{j=0}^{n-1} (a_j+b_j)^2 \mod n\\
    &\equiv& 2P + \sum_{j=0}^{n-1} a_j b_j \mod n\\
    &\equiv& 2P + \sum_{j=0}^{n-1} ((j\pi)^2-j^2) \mod n\\
    &\equiv& 2P \mod n.
\end{eqnarray*}
However, $P = n \frac{2n^2 - 3n + 1}{6}$ is not a multiple of $n$ when $n$ is a multiple of $3$, which is the desired contradiction.
\end{proof}

\subsection{Transitive groups of odd order}

Let $G$ act on $\Omega=\{1,\ldots,n\}$. An \emph{orbital} of $G$ is an orbit
of $G$ on ordered pairs. The number of orbitals is the \emph{rank} of $G$.
If $G$ is transitive on $\Omega$, there is one diagonal orbital consisting of
all pairs $(x,x)$ for $x\in \Omega$. The edge set of any $G$-invariant graph
or digraph is a union of orbitals.

There is a natural bijection between the orbitals of a transitive group and
the orbits of the stabiliser of a point: if $O$ is an orbital, the set
$\{y:(x,y)\in O\}$ is an orbit of the stabiliser of $x$.

\begin{proposition} \label{prop:orbitals}
Let $G$ be a transitive permutation group of degree $n$. Then the
permutation $\pi$ satisfies $d(g,\pi)=n-1$ for all $g\in G$ if and only
if, for every non-diagonal orbital $O$, $(x,y)\in O$ implies
$(x\pi,y\pi)\notin O$.
\end{proposition}

\begin{proof}
If $(x,y),(x\pi,y\pi)\in O$, then there exists $g\in G$ with
$(x\pi,y\pi)=(xg,yg)$, and $d(\pi,g)\le n-2$. Conversely, if not all
distances $d(g,\pi)$ are $n-1$, then (since their average is $n-1$) there
exists $g\in G$ with $d(\pi,g)\le n-2$, so $\pi$ and $g$ agree on two
distinct points $x$ and $y$, and $(x\pi,y\pi)=(xg,yg)$.
\end{proof}

\begin{corollary}
Let $G$ be a normal subgroup of a $2$-transitive group of degree $n$. Then
$r(G)=n-1$ unless $G$ is itself $2$-transitive, in which case $r(G)=n$.
\end{corollary}

\begin{proof}
If $G$ is $2$-transitive, there is only one non-diagonal orbital, and the
result follows.

Suppose that $G$ is a normal subgroup of the $2$-transitive group $H$, and
that it has $r$ non-diagonal orbitals, where $r>1$. Then $H/G$ permutes
these orbitals transitively. By Jordan's Theorem, $H/G$ contains an element
fixing no orbital. If $\pi$ is a permutation in this coset of $G$, then $\pi$
has distance $n-1$ from every element of $G$.
\end{proof}

\begin{corollary}
If $G$ is transitive of degree $n$, and a point stabiliser has an orbit
of size greater than $(n-1)/2$, then $G$ has remoteness $n$.
\end{corollary}

\begin{proof}
There is an orbital $O$ with $|O|>n(n-1)/2$; so $O\cap O\pi \ne\emptyset$
for all permutations $\pi$.
\end{proof}

Part of the following Corollary is explained by Theorem~\ref{th:odd} below.

\begin{corollary}
A permutation group of rank~$3$ has remoteness $n-1$ if and only if either
it has odd order (in which case it is a group of automorphisms of a Paley
tournament) or the graphs formed by the two non-diagonal orbitals are
isomorphic.
\end{corollary}

\begin{proof}
If the non-diagonal orbitals have different sizes, the preceding corollary
applies. If they have the same size, then $\pi$ satisfies the condition
of Proposition~\ref{prop:orbitals} if and only if it interchanges the two orbitals, i.e.
it is an isomorphism between each orbital graph and its complement.
\end{proof}

\begin{corollary} \label{cor:Paley}
If $G$ is transitive and there is a self-complementary $G$-invariant
graph,  or a self-converse $G$-invariant tournament, then $G$ has
remoteness $n-1$.
\end{corollary}

\begin{proof}
The orbitals which are edges or arcs of the given graph or tournament are
interchanged with those which are not by $\pi$.
\end{proof}

A powerful consequence is given below.

\begin{theorem} \label{th:odd}
A transitive permutation group of degree $n$ with odd order has remoteness
$n-1$.
\end{theorem}

\begin{proof}
Let $G$ be a transitive permutation group of degree $n$. By our earlier
result, it suffices to show that $G$ is contained in the automorphism group
of a self-converse tournament. We prove this by induction on $n$, so assume
that this statement holds for permutation groups of smaller degree. Note that
it suffices to prove the result in the case where $G$ is a subgroup
of $S_n$ maximal subject to having odd order.

\textbf{Case 1:} $G$ is imprimitive. By maximality, $G$ is the wreath
product of $H$ and $K$, where $H$ and $K$ are transitive groups of smaller
degree having odd order. By the inductive hypothesis, each is contained in
the automorphism group of a self-converse tournament; call these tournaments
$S$ and $T$. Then form the lexicographic product of $S$ and $T$: that is,
take $|T|$ copies of $S$ indexed by vertices of $T$, and orient edges between
two copies of $S$ according to the arc between the corresponding vertices of
$T$. Clearly $G$ is a group of automorphisms of the resulting tournament. We
have to show that it is self-converse. Choose an isomorphism $\sigma$ from
$S$ to its complement, and put a copy of $\sigma$ on each copy of $S$. Now
compose with an isomorphism $\tau$ from $T$ to its complement, blown up to
act on copies of $S$ as it does on vertices of $T$. (This blow-up procedure
means that the blown-up $\tau$ induces an isomorphism between any two copies
of $S$; combined with $\sigma$ this makes it an anti-isomorphism.)

\textbf{Case 2:} $G$ is primitive. By the Feit--Thompson theorem, $G$ is
soluble; so it has a normal elementary abelian subgroup (isomorphic to the
additive group of a vector space $V$ over a prime field) which acts regularly
on the points. So we can identify the point set with $V$. Since $G$ has odd
order, no ordered pair is reversed by $G$. So the orbits of $G$ on ordered
pairs of vertices come in self-converse pairs, and we may pick one out of
each pair to form a tournament. Since $V$ acts regularly, this tournament is
a Cayley tournament for the group $V$: that is, there is a set $S$ such that
we have arcs from $0$ to $v$ for $v\in S$, and arcs $v$ to $0$ for $v\notin S$
(and note that, if $v\in S$, then $-v\notin S$ and conversely); all other
arcs are obtained by translation. Now the linear map represented by $-I$
reverses the orientation of all edges, so induces an anti-automorphism.
\end{proof}

\paragraph{Remark} The Paley graphs are isomorphic to their complements,
so their automorphism groups have remoteness $n-1$.

The symmetric and alternating groups $S_7$ and $A_7$
acting on the set of $2$-element subsets of $\{1,\ldots,7\}$ have
rank~$3$, with two orbitals of the same size, but the two invariant graphs
(the line graph of $K_7$ and its complement) are not isomorphic. So these
groups have remoteness $n$.

\paragraph{Remark} Our results resolve the question of remoteness for many,
but not all, transitive permutation groups. So the complexity question remains
open: given permutations which generate a transitive group $G$, decide whether
$r(G)=n$.

\subsection{The remoteness graph}

Let $G$ be a group acting transitively on a set $\Omega$ with cardinality $n$.

\begin{definition}
The remoteness graph $R(G)$ of $G$ has vertex set $\Omega^2$ and two distinct ordered pairs of points $(a,b), (c,d) \in \Omega^2$ are adjacent in $R(G)$ if and only if $a =c$ or $b = d$ or there exists $g \in G$ such that $(b,d) = (ag,cg)$ (and hence $(b,d)$ and $(a,c)$ lie in the same orbital). 
\end{definition}

We easily obtain that $R(G)$ is complete if and only if $G$ is $2$-transitive. We remind the reader of the following notations from graph theory \cite{BM08}. Let $X$ be a graph, then its stability number (also known as independence number), chromatic number, and clique number are denoted as $\alpha(X)$, $\chi(X)$, and $\omega(X)$, respectively.

\begin{proposition}
For any transitive group $G$, $\alpha(R(G)) \le n$ with equality if and only if $r(G) = n-1$.
\end{proposition}

\begin{proof}
Clearly, $R(G)$ contains the Hamming graph $H(2,n)$ as a spanning subgraph, hence $\alpha \le n$. We have $\alpha = n$ if and only if there are $n$ ordered pairs $(a_i,b_i)$ such that all $a_i$ and all $b_i$ are distinct and for any $i \ne j$, $(a_i,a_j)$ is not in the same orbital as $(b_i,b_j)$. Denoting $b_i = a_i\pi$ for all $i$, we see that $\pi$ is a permutation which satisfies the conditions of Proposition~\ref{prop:orbitals}.
\end{proof}

\begin{lemma}
The graph $R(G)$ is vertex-transitive.
\end{lemma}

\begin{proof}
We prove that $G \times G$ acting coordinatewise is in the automorphism group of $R(G)$. Let $(a,b),(x,y) \in \Omega^2$; we have $x = ag, y=bh$ for some $g,h \in G$. Therefore, consider two vertices $(c,d), (e,f) \in \Omega^2$; they are adjacent if and only if either $c=e$, $d=f$, or $(c,e) = (dg_1,fg_1)$ for some $g_1 \in G$. This is equivalent to either $cg = eg$, $dh=fh$, or $(cg,eg) = (dhg_2,fhg_2)$ where $g_2 = h^{-1}g_1g$; in other words, $(cg,dh)$ is adjacent to $(eg,fh)$.
\end{proof}


We hence have the following inequalities \cite[Corollary 7.5.2]{GR01}:
\begin{equation} \nonumber
    n \le \omega(R(G)) \le \frac{n^2}{\alpha(R(G))} \le \chi(R(G)),
\end{equation}
and we want to know when equality holds for the first two. Note that if $G$ has a subset of $n$ permutations with minimum distance $n$ (e.g., a regular subgroup), then we require that not only $\alpha(R(G)) = n$, but also that the whole vertex set be partitioned into $n$ stable sets. In other words, if $G$ has a regular subgroup, then it has remoteness $n-1$ if and only if $\chi(\overline{R(G)}) = \alpha(R(G)) = \omega(R(G)) = n$.

Since $R(G)$ is vertex-transitive, it is regular, and its valency can be easily computed. We have $(a,a) \sim (b,c)$ if and only if either $b=a$, $c=a$, or $(a,b)$ and $(a,c)$ are in the same non-diagonal orbital. For any non-diagonal orbital $O$, we have $\left(\frac{|O|}{n}\right)^2$ neighbours of $(a,a)$ from $O$. Therefore, the valency is given by
\begin{equation} \nonumber
    2(n-1) + \frac{1}{n^2} \sum_{O\, \mbox{non-diagonal}} |O|^2.
\end{equation}

We can define an analogous graph for any set of permutations. If the set is a Latin square (in particular if it is a regular permutation group), then the graph is the strongly regular {\em Latin square graph} \cite{BC03} with parameters $(n^2,3(n-1),n,6)$; its clique number is $n$ (if $n > 2$), its stability number is $n$ if and only if the Latin square has a transversal, and its chromatic number is $n$ if and only if the Latin square has an orthogonal mate.

\subsection{List of transitive groups with remoteness $n-1$}

The table gives all transitive groups of degree $n<10$ which have
remoteness $n-1$. The first column gives the degree; the second column
the number in the \textsf{GAP} listing (so that the \textsf{GAP} command
\texttt{TransitiveGroup(9,17)} produces the last group in the list, for
example); the third column the order of the group; and the fourth column
refers to a note
giving a result in our paper justifying the conclusion where possible.
There are no transitive groups of degree~$6$ with remoteness~$5$; for
the last three groups of degree~$8$, the result is shown by computation.

\[\begin{array}{|r|r|r|r|}
\hline
n & k & |G| & \hbox{Note} \\\hline
3 & 1 & 3 & 1 \\\hline
4 & 2 & 4 & 2 \\\hline
5 & 1 & 5 & 1 \\
  & 2 & 10 & 3 \\\hline
7 & 1 & 7 & 1 \\
  & 2 & 14 & 3 \\
  & 3 & 21 & 4 \\\hline
8 & 2 & 8 & 2 \\
  & 3 & 8 & 2 \\
  & 4 & 8 & 2 \\
  & 5 & 8 & 2 \\
  & 9 & 16 & \\
  & 10 & 16 & \\
  & 11 & 16 & \\\hline
9 & 1 & 9 & 1 \\
  & 2 & 9 & 4 \\
  & 4 & 18 & 5 \\
  & 5 & 18 & 5 \\
  & 6 & 27 & 4 \\
  & 7 & 27 & 4 \\
  & 8 & 36 & 5 \\
  & 9 & 36 & 5 \\
  & 16 & 72 & 5 \\
  & 17 & 81 & 4 \\\hline
\end{array}\]

\paragraph{Notes}
\begin{itemize}
\item[1] Cyclic group; Theorem \ref{th:one_generator}
\item[2] Regular group, non-cyclic Sylow $2$-subgroup; Corollary \ref{cor:2-transitive}
\item[3] Dihedral group; Proposition \ref{prop:dihedral}
\item[4] Group of odd order; Theorem \ref{th:odd}
\item[5] Automorphism group of Paley graph; Corollary \ref{cor:Paley}
\end{itemize}

\bibliographystyle{model1b-num-names}
\bibliography{r}

\end{document}